\documentclass[11pt]{amsart}          
\usepackage{amsmath, amssymb, amsfonts, amsthm, bm}

\newtheorem{theorem}{Theorem}

\newtheorem*{Auerlm}{Auerbach's Lemma}
\newtheorem*{BMineq}{Brunn-Minkowski inequality}

\newcommand{\sgn}{\operatorname{sgn}}
\newcommand{\vol}{\operatorname{vol}}
\newcommand{\reals}{\mathbb{R}}
\newcommand{\real}[1]{\reals^{#1}}
\newcommand{\fn}[3]{{#1}:{#2}\rightarrow{#3}}
\newcommand{\set}[2]{\{ {#1} \:|\: {#2} \}}
\newcommand{\ipr}[2]{\langle {#1}, {#2} \rangle}
\newcommand{\vect}[1]{\bm{#1}}
\newcommand{\ve}{\vect{e}}
\newcommand{\vu}{\vect{u}}
\newcommand{\vv}{\vect{v}}
\newcommand{\vx}{\vect{x}}
\newcommand{\vy}{\vect{y}}
\newcommand{\vz}{\vect{z}}

\begin{document}
\title[Extremal Problems in Minkowski Space]{Extremal Problems in Minkowski Space related to Minimal Networks}
\author{K.\ J.\ Swanepoel}
\address{Department of Mathematics and Applied Mathematics 
\\ University of Pretoria \\ Pretoria 0002 \\ South Africa}
\email{konrad@math.up.ac.za}
\begin{abstract}
We solve the following problem of Z.\ F\"uredi, J.\ C.\ Lagarias
and F.\ Morgan \cite{FLM}:
Is there an upper bound polynomial in $n$ for the largest cardinality of a set
$S$ of unit vectors in an $n$-dimensional Minkowski space (or Banach space)
such that the sum of any subset has norm less than 1?
We prove that $|S|\leq 2n$ and that equality holds iff the space is linearly
isometric to $\ell^{n}_{\infty}$, the space with an $n$-cube as unit ball.
We also remark on similar questions raised in \cite{FLM} that arose out of
the study of singularities in length-minimizing networks in Minkowski spaces.
\end{abstract}
\maketitle
\section{Introduction}
In \cite{LM} Lawlor and Morgan derived a geometrical description for the
singularities (Steiner points) of a length-minimizing network connecting a 
finite set of points in a smooth Minkowski space (finite dimensional
Banach space).
In Euclidean space the geometrical description is equivalent to the classical
result that at a singularity three line segments meet at $120^{\circ}$
angles.
See also \cite{BG}, \cite{M} and \cite{CR} for a discussion of
length-minimizing networks and their history. The geometrical description of 
Lawlor and Morgan leads to extremal problems of a combinatorial type in
strictly convex Minkowski spaces.
Such problems are considered in \cite{FLM}.
In this note we briefly remark on some of these problems and solve one of the 
open problems stated in \cite{FLM} (see theorem \ref{mainth}).

\section{Preliminaries}
We denote the real numbers by $\reals$ and the real vector space of
$n$-tuples of real numbers by $\real{n}$.
The coordinates of a vector $\vx \in \real{n}$ will be denoted by
$\vx = (\vx(1), \vx(2), \dots, \vx(n))$.
The standard basis $\ve_{1}, \ve_{2}, \dots, \ve_{n}$ will be
used, where $\ve_{i}$ is the vector for which $\ve_{i}(i) = 1$ and 
$\ve_{i}(j) = 0$ for $i \neq j$.
A {\it Minkowski space} (or finite-dimensional Banach space) $(\real{n},
\Phi)$ is $\real{n}$ endowed with a norm $\Phi$.
A Minkowski space is {\it strictly convex} if $\Phi(\vx)=\Phi(\vy) 
=1, \vx\neq \vy$ implies $\Phi(\vx+\vy) < 2$.

We denote by $\ell^{n}_{p}$ the $n$-dimensional Minkowski space with norm
  $$\Phi_{p}(\vx) = \left(\sum_{i=1}^{n} |\vx(i)|^{p}\right)^{1/p}$$
for $p \geq 1$, and by $\ell^{n}_{\infty}$ the space with norm
  $$\Phi_{\infty}(\vx) = \max_{1\leq i\leq n} |\vx(i)|.$$

We now state Auerbach's lemma which relates the spaces $\ell^{n}_{1}$ and
$\ell^{n}_{\infty}$ to an arbitrary Minkowski space in $n$ dimensions.
A proof may be found in \cite[page 29]{Pietsch}.
                                       
\begin{Auerlm}
  For any Minkowski space $(\real{n},\Phi)$ there
  exists a linear isomorphism $\fn{T}{\real{n}}{\real{n}}$ such that
  $\Phi_{\infty}(\vx) \leq \Phi(T\vx) \leq \Phi_{1}(\vx)$,
  i.e.\
  \begin{equation}
    \label{eq:auerbach}
    \max_{1\leq i\leq n} |\vx(i)| \leq \Phi(T\vx)
     \leq \sum_{i=1}^{n} |\vx(i)|.
  \end{equation}
\end{Auerlm}

We denote the $n$-dimensional Lebesgue measure (or {\it volume}) of
measurable $V\subseteq\real{n}$ by $\vol(V)$.
If $U,V\subseteq\real{n}$, then we define $U+V=\set{\vu+\vv}{\vu\in U,\vv\in V}$.
The Brunn-Minkowski inequality relates the volumes of compact $U$
and $V$ to that of $U+V$.
A proof may be found in \cite{BZ}.

\begin{BMineq}
If $U,V\subseteq\real{n}$ are compact, then
$$(\vol(U+V))^{1/n} \geq (\vol(U))^{1/n}+(\vol(V))^{1/n}.$$
\end{BMineq}

\section{Extremal Problems}           
  \label{sec3}
From now on $S$ will denote a finite set of unit vectors in a Minkowski space.
In \cite{FLM} the following type of extremal problems is considered:
Find the largest cardinality of $S$ satisfying a selection of the following
conditions:
$$\begin{array}{lcl}
  (A) & \Phi(\sum_{\vx\in J}\vx) \leq 1 \text{ for all } J
  \subseteq S & \text{(the {\it strong collapsing condition})}  \\[2pt]
  (A') & \Phi(\vx + \vy) \leq 1\text{ for all } \vx,\vy\in S, 
   \vx\neq \vy & \text{(the {\it weak collapsing condition})} \\[2pt]
  (B) & \sum_{\vx\in S}\vx = \bm{0}
  & \text{(the {\it strong balancing condition})} \\[2pt]
  (B') & \begin{array}[c]{c}
           \bm{0} \text{ is in the relative interior } \\
           \text{of the convex hull of } S
         \end{array} & \text{(the {\it weak balancing condition})}
\end{array}$$       
See \cite{LM} and \cite{FLM} for the connection between these conditions and 
minimal networks.
In \cite{FLM} it is proved that $(A')$ and $(B')$ together give an upper bound
$|S| \leq 2n$ for an arbitrary Minkowski space, and $|S| \leq n+1$ for
strictly convex Minkowski spaces.
In \cite{LM} it is proved that there exist a strictly convex norm on
$\real{n}$ and a subset $S$ of $n+1$ unit vectors satisfying $(A)$ and $(B)$.
$|S| = 2n$ is attained in for example $\ell^{n}_{\infty}$ with 
$S= \set{\pm\ve_{i}}{1\leq i\leq n}$, in which case even the strong 
conditions $(A)$ and $(B)$ hold.
However, there are other Minkowski spaces where equality is also attained 
(see theorem \ref{th1}).
This is to be contrasted with theorem \ref{mainth} where we show that the
extreme case for $S$ satisfying $(A)$ and $(B)$ can only be attained for 
$\ell^{n}_{\infty}$.
\begin{theorem}
\label{th1}
  For infinitely many $n\geq1$ there exists a set of unit vectors
  $S = \{\vx_{1},\dots, \vx_{2n}\} \subseteq \ell^{n}_{1}$ satisfying
  $(A')$ and the strong balancing condition $(B)$.
  In particular, such a set exists if a Hadamard matrix of order $n$ exists.
\end{theorem}
\begin{proof}
We recall that an $n\times n$ Hadamard matrix $H$ consists of $(\pm 1)$-entries
such that $HH^{t} = nI$, and such matrices exist for infinitely many $n$ (see
\cite[chapter 18]{vLW}).
We let $\vv_{1},\dots,\vv_{n}$ be the column vectors of $H$, and
set $\vx_{i}=\frac{1}{n}\vv_{i}$ for $i=1,\dots,n$.
Then $S:= \set{\pm\vx_{i}}{1\leq i\leq n}$ is a set of $2n$ unit vectors.
Since the column vectors of $H$ are orthogonal, $\ipr{\vv_{i}}{\vv_{j}} = 0$ for $i\neq j$, implying that
$\Phi_{1}(\vx_{i} + \vx_{j}) = 1$ and $\Phi_{1}(\vx_{i}
- \vx_{j}) = 1$ for all $i \neq j$.
It follows that $S$ satisfies $(A')$ and $(B)$.
\end{proof}

The question now is what happens if there is no balancing condition present.
In \cite{FLM} an upper bound of $|S| < 3^{n}$ is derived from the weak
collapsing condition $(A')$ alone using a volume argument.
Using the Brunn-Minkowski inequality we obtain a sharper bound (theorem 
\ref{th2}).
In \cite{FLM} a strictly convex norm and a set $S$ of unit vectors with 
$|S| \geq (1.02)^{n}$ satisfying $(A')$ are constructed for all sufficiently 
large $n$.
It would be interesting to find the greatest lower bound of the $\alpha$'s
for which $|S|\leq\alpha^{n}$ for any set $S$ of unit vectors in an 
arbitrary Minkowski space satisfying $(A')$, and sufficiently large $n$.

\begin{theorem}
  \label{th2}
  If a set $S$ of unit vectors in $\real{n}$ satisfy $(A')$, then $|S| <
    2^{n+1}$.
\end{theorem}
\begin{proof}
We denote the closed unit ball with centre $\vx$ and radius $r$ by
$B(\vx,r) = \set{\vy\in\real{n}}{\Phi(\vx - \vy)\leq r}$, and the volume 
of a ball of unit radius by $\beta$.
For distinct $\vx,\vy\in S$ we obtain from the triangle inequality
that $\Phi(\vx-\vy) \geq 1$.
Let $k=|S|$.
We partition $S$ into two sets $S_{1}$ and $S_{2}$ of sizes
$\lfloor{k/2}\rfloor$ and $\lceil{k/2}\rceil$, respectively.
Let $V_{i} = B(\bm{0},\frac{1}{2}) \cup \bigcup_{\vx\in S_{i}} B(\vx,\frac{1}{2})$ for $i=1,2$.
Clearly, each $V_{i}$ consists of closed balls with disjoint interiors, and
therefore, $\vol(V_{1})=\beta(\lfloor k/2\rfloor+1)2^{-n}$ and
$\vol(V_{2})=\beta(\lceil k/2\rceil + 1)2^{-n}$.
Using $(A')$ we obtain $V_{1}+ V_{2}\subseteq B(\bm{0},2)$, and
$\vol(V_{1}+V_{2})\leq 2^{n}\beta$.
By the Brunn-Minkowski inequality we now have
$$2\beta^{1/n} \geq {\textstyle \frac{1}{2}}\beta^{1/n}
  (\lfloor k/2\rfloor+1)^{1/n} +
   {\textstyle \frac{1}{2}}\beta^{1/n}(\lceil k/2\rceil+1)^{1/n} > 
\beta^{1/n}(k/2)^{1/n},$$
and $|S| < 2^{n+1}$.
\end{proof}

From the above proof we actually find that if $\Phi(\vx) = \Phi(\vy)
= 1$ and $\Phi(\vx+\vy) \leq 1$ imply $\Phi(\vx-\vy) \geq r 
> 1$, then $|S|\leq 2(1+1/r)^{n}+1$ for $S$ satisfying $(A')$.
Such is the case for $\ell^{n}_{p}$: It follows from the Clarkson inequality
\cite{C} for $p\geq 2$, and the Hanner inequality \cite{H} for 
$1<p<2$, that $r$ may be taken to be $3^{1/p}$ for $p\geq 2$, and
$(2^{p}-1)^{1/p}$ for $1<p<2$.

For $\ell^{n}_{1}$ an upper bound $|S|\leq 2^n$ holds:
If the coordinates of two unit vectors $\vx$ and $\vy$ have
the same sequence of signs, i.e.\ $\sgn(\vx(i)) = \sgn(\vy(i))$ for 
all $i=1,\dots,n$, then $\Phi_{1}(\vx+\vy) =2$, contradicting $(A')$.
In the Euclidean case $\ell^{n}_{2}$ we of course have $|S| \leq 3$, 
independent of $n$.
For $\ell^{n}_{\infty}$ the sharp upper bound $|S| \leq 2n$ holds:
If $|S| \geq 2n+1$, then by the pigeon-hole principle there are three vectors
$\vx,\vy,\vz\in S$ and an $i\in\{1,\dots,n\}$ such that $|\vx(i)| 
= |\vy(i)| = |\vz(i)| = 1$.
Some two of these vectors will have the same sign in the $i$th coordinate,
and their sum will then have a norm of 2.

In \cite[problem 3.7]{FLM} the question is asked whether the strong
collapsing condition $(A)$ on its own gives an upper bound for $|S|$ that is 
polynomial in $n$.                                                    
A linear upper bound may be derived by the same technique as in theorem 
\ref{th2}.
We partition the elements of $S$ except for at most 2 into subsets
$S_{1},\dots,S_{k}$ of size 3, where $k=\lfloor |S|/3 \rfloor$.
For $i=1,\dots,k$ let
 $$V_{i}=\bigcup_{\vx\in S_{i}} B(\vx,{\textstyle\frac{1}{2}}) \cup
\bigcup_{\vx,\vy\in S_{i},\vx\neq \vy} B(\vx+\vy,{\textstyle\frac{1}{2}}).$$
From $(A)$ it follows that each $V_{i}$ consists of 6 balls with disjoint 
interiors, and $V_{1}+\dots+V_{k}\subseteq B(\bm{0},\frac{1}{2}k+1)$.
By the Brunn-Minkowski inequality we obtain $\frac{1}{2}k+1\geq
\frac{1}{2}6^{1/n}k$, and $k\leq 2/(6^{1/n}-1)$.
Therefore, $|S|\leq 6/(6^{1/n}-1)+2 < (6/\ln 6)n$, after some calculus.
This bound is not sharp, however.
In the following theorem we derive the sharp upper bound $|S|\leq 2n$.

\begin{theorem}
\label{mainth}
Let $S$ be a finite set of unit vectors in a Minkowski space
$(\real{n},\Phi)$ satisfying the collapsing condition $(A)$.
Then $|S| \leq 2n$, and equality holds iff $(\real{n},\Phi)$ is linearly
isometric to $\ell^{n}_{\infty}$, with $S$ corresponding to the set
$\set{\pm \ve_{i}}{1\leq i\leq n}$ under any isometry.
\end{theorem}
\begin{proof}
By Auerbach's lemma we may assume (after applying a linear isomorphism
of $\real{n}$) that for any vector $\vx\in \real{n}$ the inequalities 
\eqref{eq:auerbach} hold, with $T$ now the identity.
Choose $m$ distinct vectors $\vx_{1},\dots,\vx_{m}$ from $S$.
By \eqref{eq:auerbach} we have
\begin{equation}            
\label{ineq1}
  \sum_{i=1}^{n} |\vx_{j}(i)| \geq 1 \text{ for all } j=1,\dots m.
\end{equation}
Suppose that for some coordinate $i\in\{1,\dots,n\}$ we have
$$\sum_{\substack{j=1\\ \vx_{j}(i) \geq 0}}^{n} \negmedspace\vx_{j}(i) > 1.$$
Then
$$ \Phi(\negmedspace\sum_{\substack{j=1 \\ \vx_{j}(i) \geq 0}}^{m}
        \negmedspace\vx_{j}) \geq
        \sum_{\substack{j=1 \\ \vx_{j}(i) \geq 0}}^{m} 
        \negmedspace\vx_{j}(i)$$
by \eqref{eq:auerbach}, contradicting $(A)$.
Therefore,
\begin{equation}
  \label{ineq2}
  \sum_{\substack j=1 \\ \vx_{j}(i) \geq 0}^{m}
    \negthickspace \vx_{j}(i)
   \leq 1 \text{ for all } i=1,\dots n,
\end{equation}
and similarly,
\begin{equation}
  \label{ineq3}
  \sum_{\substack{j=1 \\ \vx_{j}(i) \leq 0}}^{m}
     \negthickspace -\vx_{j}(i)
  \leq 1 \text{ for all } i=1,\dots,n.
\end{equation}
From \eqref{ineq2} and \eqref{ineq3} it follows that $\sum_{j=1}^{m}|\vx_{j}(i)| \leq 2$ for all $i=1,\dots,n$, and from \eqref{ineq1} we have
\begin{equation}
\label{ineq4}
  m \leq \sum_{j=1}^{m}\sum_{i=1}^{n} |\vx_{j}(i)| \leq 2n,
\end{equation}
and $|S| \leq 2n$.

If $|S| = 2n$ for some set of unit vectors
 $S=\{\vx_{1},\dots,\vx_{2n}\}$ satisfying $(A)$, then equality must 
hold in \eqref{ineq4}, \eqref{ineq2} and \eqref{ineq3}.
Therefore, $\sum_{j=1}^{2n}\vx_{j} = 0$, showing that in the extreme
case the strong balancing condition $(B)$ must be satisfied.
We now show that conditions $(A)$ and $(B)$ together with the assumption
$|S|=2n$ imply that $(\real{n},\Phi)$
is linearly isometric to $\ell^{n}_{\infty}$, and $S$ corresponds to 
$\set{\pm\ve_{i}}{1\leq i\leq n}$, as claimed in \cite{FLM}.

We recall theorem 3.1 of \cite{FLM}:
\begin{quote}
\it
  If $(\real{n},\Phi)$ is a Minkowski space and $S$ is a set of unit vectors
  satisfying $(A')$ and $(B')$, then $|S|\leq 2n$, and if equality holds, then 
  $S$ corresponds to $\set{\pm\ve_{i}}{1\leq i\leq n}$ under some linear
  isomorphism.
\end{quote}
We therefore have $S=\set{\pm\vx_{i}}{1\leq i\leq n}$ where the 
$\vx_{i}$'s are linearly independent.
We first show that if $(\real{n},\Phi) = \ell^{n}_{\infty}$ then 
$S=\set{\pm\ve_{i}}{1\leq i\leq n}$ must hold.
For $i=1,\dots,n$ choose $j_{i}\in\{1,\dots,n\}$ such that $|\vx_{i}(j_{i})| = 1$.
After renaming, we may assume $\vx_{i}(j_{i})=1$.
The $j_{i}$'s must be distinct, otherwise $(A)$ is contradicted.
We may therefore rename the $\vx_{i}$'s to obtain $\vx_{i}(i)=1$ for 
$i=1,\dots,n$.
If we have $\vx_{i}(j) \neq 0$ for some $i\neq j$, then either 
$\Phi_{\infty}(\vx_{i}+\vx_{j}) > 1$ or $\Phi_{\infty}(-\vx_{i}+\vx_{j}) > 1$, contradicting $(A)$.
Therefore, $\vx_{i}(j)=0$ for all $i\neq j$, and we have 
 $S=\set{\pm\ve_{i}}{i=1,\dots,n}$.

To show that in fact $(\real{n},\Phi)$ is linearly isometric to 
$\ell^{n}_{\infty}$, we use the following theorem of Petty \cite{Petty} 
(see also \cite[theorem 2.1]{FLM}):
\begin{quote}
\it
  If $T$ is a subset of a Minkowski space $(\real{n},\Phi)$ such that
  $\Phi(\vx-\vy)=1$ for all $\vx,\vy\in T, \vx\neq \vy$, 
  then $|T|\leq 2^n$, with equality iff $(\real{n},\Phi)$ is linearly
  isometric to $\ell^{n}_{\infty}$.
\end{quote}
We will apply this theorem to the set $T = \set{\sum_{i\in A}\vx_{i}}{A
\subseteq \{1,\dots,n\}}$.
Obviously $|T|=2^n$.
We now show that
\begin{equation}
\label{eq1}
 \Phi(\sum_{i\in A} \vx_{i} - \sum_{i\in B} \vx_{i})
= 1 \text{ for all } A,B\subseteq \{1,\dots,n\} \text{, } A\neq B,
\end{equation}
thus completing the proof.
Firstly, we have
  $$\Phi(\sum_{i\in A} \vx_{i} - \sum_{i\in B} \vx_{i})
     = \Phi(\sum_{i\in A\setminus B} \negmedspace\vx_{i}
     + \sum_{i\in B\setminus A} \negmedspace -\vx_{i})
    \leq 1$$
by $(A)$.
Secondly, $A\setminus B \neq \emptyset$ or $B\setminus A \neq \emptyset$,
since $A\neq B$.
We assume without loss that $A\setminus B\neq\emptyset$ and choose $j\in
A\setminus B$.
Then
\begin{align*}
  2 = \Phi(2\vx_{j}) 
    & \leq  \Phi(\sum_{i\in A} \vx_{i} - \sum_{i\in B} \vx_{i})
          + \Phi(\sum_{i\in B\cup\{j\}} \!\negthickspace \vx_{i}
                  - \sum_{i\in A\setminus\{j\}} \!\negthickspace\vx_{i}) \\
    & \leq  \Phi(\sum_{i\in A} \vx_{i} - \sum_{i\in B} \vx_{i}) + 1,
\end{align*}
showing that \eqref{eq1} holds.
\end{proof}

For strictly convex norms the bound in the above theorem should perhaps be 
$|S|\leq n+1$, but this seems to require a new idea.

\end{document}